\documentclass[11pt, a4paper]{amsart}    
\usepackage{latexsym, amsmath, amsthm, amssymb, setspace,verbatim}
\usepackage[all]{xy}
\usepackage{hyperref}
\usepackage{enumitem}
\usepackage{nicefrac}

\title[Approaching the UCT problem via crossed products]{Approaching the UCT problem via crossed products of the Razak--Jacelon algebra}
\author{Sel\c{c}uk Barlak}
\author{G{\'a}bor Szab{\'o}}

\address{Europa-Universit\"at Flensburg,
Institut f\"ur mathematische, 
naturwissen\-schaftliche und technische Bildung,
Abteilung f\"ur Mathematik und ihre Didaktik,
Auf dem Campus 1b,
DE-24943 Flensburg,
Germany}
\email{selcuk.barlak@uni-flensburg.de}

\address{Department of Mathematics, KU Leuven, Celestijnenlaan 200b, box 2400, B-3001 Leuven, Belgium}
\email{gabor.szabo@kuleuven.be}

\thanks{The first author was supported by the Villum Fonden project grant 'Local and global structures of groups and their algebras' (2014-2018)}
\thanks{The second author was supported by EPSRC grant EP/N00874X/1, the Danish National Research Foundation through the Centre for Symmetry and Deformation (DNRF92), and the European Union's Horizon 2020 research and innovation programme under the Marie Sklodowska-Curie grant agreement 746272}
\subjclass[2010]{Primary 46L35, 46L55}

\begin{document}

\renewcommand\matrix[1]{\left(\begin{array}{*{10}{c}} #1 \end{array}\right)}  
\newcommand\set[1]{\left\{#1\right\}}  
\newcommand\mset[1]{\left\{\!\!\left\{#1\right\}\!\!\right\}}

\newcommand{\IA}[0]{\mathbb{A}} \newcommand{\IB}[0]{\mathbb{B}}
\newcommand{\IC}[0]{\mathbb{C}} \newcommand{\ID}[0]{\mathbb{D}}
\newcommand{\IE}[0]{\mathbb{E}} \newcommand{\IF}[0]{\mathbb{F}}
\newcommand{\IG}[0]{\mathbb{G}} \newcommand{\IH}[0]{\mathbb{H}}
\newcommand{\II}[0]{\mathbb{I}} \renewcommand{\IJ}[0]{\mathbb{J}}
\newcommand{\IK}[0]{\mathbb{K}} \newcommand{\IL}[0]{\mathbb{L}}
\newcommand{\IM}[0]{\mathbb{M}} \newcommand{\IN}[0]{\mathbb{N}}
\newcommand{\IO}[0]{\mathbb{O}} \newcommand{\IP}[0]{\mathbb{P}}
\newcommand{\IQ}[0]{\mathbb{Q}} \newcommand{\IR}[0]{\mathbb{R}}
\newcommand{\IS}[0]{\mathbb{S}} \newcommand{\IT}[0]{\mathbb{T}}
\newcommand{\IU}[0]{\mathbb{U}} \newcommand{\IV}[0]{\mathbb{V}}
\newcommand{\IW}[0]{\mathbb{W}} \newcommand{\IX}[0]{\mathbb{X}}
\newcommand{\IY}[0]{\mathbb{Y}} \newcommand{\IZ}[0]{\mathbb{Z}}

\newcommand{\CA}[0]{\mathcal{A}} \newcommand{\CB}[0]{\mathcal{B}}
\newcommand{\CC}[0]{\mathcal{C}} \newcommand{\CD}[0]{\mathcal{D}}
\newcommand{\CE}[0]{\mathcal{E}} \newcommand{\CF}[0]{\mathcal{F}}
\newcommand{\CG}[0]{\mathcal{G}} \newcommand{\CH}[0]{\mathcal{H}}
\newcommand{\CI}[0]{\mathcal{I}} \newcommand{\CJ}[0]{\mathcal{J}}
\newcommand{\CK}[0]{\mathcal{K}} \newcommand{\CL}[0]{\mathcal{L}}
\newcommand{\CM}[0]{\mathcal{M}} \newcommand{\CN}[0]{\mathcal{N}}
\newcommand{\CO}[0]{\mathcal{O}} \newcommand{\CP}[0]{\mathcal{P}}
\newcommand{\CQ}[0]{\mathcal{Q}} \newcommand{\CR}[0]{\mathcal{R}}
\newcommand{\CS}[0]{\mathcal{S}} \newcommand{\CT}[0]{\mathcal{T}}
\newcommand{\CU}[0]{\mathcal{U}} \newcommand{\CV}[0]{\mathcal{V}}
\newcommand{\CW}[0]{\mathcal{W}} \newcommand{\CX}[0]{\mathcal{X}}
\newcommand{\CY}[0]{\mathcal{Y}} \newcommand{\CZ}[0]{\mathcal{Z}}

\newcommand{\FA}[0]{\mathfrak{A}} \newcommand{\FB}[0]{\mathfrak{B}}
\newcommand{\FC}[0]{\mathfrak{C}} \newcommand{\FD}[0]{\mathfrak{D}}
\newcommand{\FE}[0]{\mathfrak{E}} \newcommand{\FF}[0]{\mathfrak{F}}
\newcommand{\FG}[0]{\mathfrak{G}} \newcommand{\FH}[0]{\mathfrak{H}}
\newcommand{\FI}[0]{\mathfrak{I}} \newcommand{\FJ}[0]{\mathfrak{J}}
\newcommand{\FK}[0]{\mathfrak{K}} \newcommand{\FL}[0]{\mathfrak{L}}
\newcommand{\FM}[0]{\mathfrak{M}} \newcommand{\FN}[0]{\mathfrak{N}}
\newcommand{\FO}[0]{\mathfrak{O}} \newcommand{\FP}[0]{\mathfrak{P}}
\newcommand{\FQ}[0]{\mathfrak{Q}} \newcommand{\FR}[0]{\mathfrak{R}}
\newcommand{\FS}[0]{\mathfrak{S}} \newcommand{\FT}[0]{\mathfrak{T}}
\newcommand{\FU}[0]{\mathfrak{U}} \newcommand{\FV}[0]{\mathfrak{V}}
\newcommand{\FW}[0]{\mathfrak{W}} \newcommand{\FX}[0]{\mathfrak{X}}
\newcommand{\FY}[0]{\mathfrak{Y}} \newcommand{\FZ}[0]{\mathfrak{Z}}

\newcommand{\Fa}[0]{\mathfrak{a}} \newcommand{\Fb}[0]{\mathfrak{b}}
\newcommand{\Fc}[0]{\mathfrak{c}} \newcommand{\Fd}[0]{\mathfrak{d}}
\newcommand{\Fe}[0]{\mathfrak{e}} \newcommand{\Ff}[0]{\mathfrak{f}}
\newcommand{\Fg}[0]{\mathfrak{g}} \newcommand{\Fh}[0]{\mathfrak{h}}
\newcommand{\Fi}[0]{\mathfrak{i}} \newcommand{\Fj}[0]{\mathfrak{j}}
\newcommand{\Fk}[0]{\mathfrak{k}} \newcommand{\Fl}[0]{\mathfrak{l}}
\newcommand{\Fm}[0]{\mathfrak{m}} \newcommand{\Fn}[0]{\mathfrak{n}}
\newcommand{\Fo}[0]{\mathfrak{o}} \newcommand{\Fp}[0]{\mathfrak{p}}
\newcommand{\Fq}[0]{\mathfrak{q}} \newcommand{\Fr}[0]{\mathfrak{r}}
\newcommand{\Fs}[0]{\mathfrak{s}} \newcommand{\Ft}[0]{\mathfrak{t}}
\newcommand{\Fu}[0]{\mathfrak{u}} \newcommand{\Fv}[0]{\mathfrak{v}}
\newcommand{\Fw}[0]{\mathfrak{w}} \newcommand{\Fx}[0]{\mathfrak{x}}
\newcommand{\Fy}[0]{\mathfrak{y}} \newcommand{\Fz}[0]{\mathfrak{z}}

\newcommand{\Ra}[0]{\Rightarrow}
\newcommand{\La}[0]{\Leftarrow}
\newcommand{\LRa}[0]{\Leftrightarrow}

\renewcommand{\phi}[0]{\varphi}
\newcommand{\eps}[0]{\varepsilon}

\newcommand{\quer}[0]{\overline}
\newcommand{\uber}[0]{\choose}
\newcommand{\ord}[0]{\operatorname{ord}}	
\newcommand{\GL}[0]{\operatorname{GL}}
\newcommand{\supp}[0]{\operatorname{supp}}	
\newcommand{\id}[0]{\operatorname{id}}		
\newcommand{\Sp}[0]{\operatorname{Sp}}		
\newcommand{\eins}[0]{\mathbf{1}}			
\newcommand{\diag}[0]{\operatorname{diag}}
\newcommand{\ind}[0]{\operatorname{ind}}
\newcommand{\auf}[1]{\quad\stackrel{#1}{\longrightarrow}\quad}
\newcommand{\hull}[0]{\operatorname{hull}}
\newcommand{\prim}[0]{\operatorname{Prim}}
\newcommand{\ad}[0]{\operatorname{Ad}}
\newcommand{\quot}[0]{\operatorname{Quot}}
\newcommand{\ext}[0]{\operatorname{Ext}}
\newcommand{\ev}[0]{\operatorname{ev}}
\newcommand{\fin}[0]{{\subset\!\!\!\subset}}
\newcommand{\diam}[0]{\operatorname{diam}}
\newcommand{\Hom}[0]{\operatorname{Hom}}
\newcommand{\Aut}[0]{\operatorname{Aut}}
\newcommand{\Ext}[0]{\operatorname{Ext}}
\newcommand{\del}[0]{\partial}
\newcommand{\dimeins}[0]{\dim^{\!+1}}
\newcommand{\dimnuc}[0]{\dim_{\mathrm{nuc}}}
\newcommand{\dimnuceins}[0]{\dimnuc^{\!+1}}
\newcommand{\dr}[0]{\operatorname{dr}}
\newcommand{\dimrok}[0]{\dim_{\mathrm{Rok}}}
\newcommand{\dimrokeins}[0]{\dimrok^{\!+1}}
\newcommand{\dreins}[0]{\dr^{\!+1}}
\newcommand*\onto{\ensuremath{\joinrel\relbar\joinrel\twoheadrightarrow}} 
\newcommand*\into{\ensuremath{\lhook\joinrel\relbar\joinrel\rightarrow}}  
\newcommand{\im}[0]{\operatorname{im}}
\newcommand{\dst}[0]{\displaystyle}
\newcommand{\cstar}[0]{\ensuremath{\mathrm{C}^*}}
\newcommand{\ann}[0]{\operatorname{Ann}}
\newcommand{\dist}[0]{\operatorname{dist}}
\newcommand{\idlat}[0]{\operatorname{IdLat}}
\newcommand{\Cu}[0]{\operatorname{Cu}}
\newcommand{\Ost}[0]{\CO_\infty^{\mathrm{st}}}
\newcommand{\linhull}{\operatorname{span}}
\newcommand{\ue}{{\ensuremath{\approx_{\mathrm{u}}}}}

\newtheorem{satz}{Satz}[section]		
\newtheorem{cor}[satz]{Corollary}
\newtheorem{lemma}[satz]{Lemma}
\newtheorem{prop}[satz]{Proposition}
\newtheorem{theorem}[satz]{Theorem}
\newtheorem*{theoremoz}{Theorem}

\theoremstyle{definition}
\newtheorem{defi}[satz]{Definition}
\newtheorem*{defioz}{Definition}
\newtheorem{defprop}[satz]{Definition \& Proposition}
\newtheorem{nota}[satz]{Notation}
\newtheorem*{notaoz}{Notation}
\newtheorem{rem}[satz]{Remark}
\newtheorem*{remoz}{Remark}
\newtheorem{example}[satz]{Example}
\newtheorem{defnot}[satz]{Definition \& Notation}
\newtheorem{question}[satz]{Question}
\newtheorem*{questionoz}{Question}
\newtheorem{construction}[satz]{Construction}
\newtheorem{conjecture}[satz]{Conjecture}


\begin{abstract}
We show that the UCT problem for separable, nuclear \cstar-algebras relies only on whether the UCT holds for crossed products of certain finite cyclic group actions on the Razak--Jacelon algebra.
This observation is analogous to and in fact recovers a characterization of the UCT problem in terms of finite group actions on the Cuntz algebra $\CO_2$ established in previous work by the authors.
Although based on a similar approach, the new conceptual ingredients in the finite context are the recent advances in the classification of stably projectionless \cstar-algebras, as well as a known characterization of the UCT problem in terms of certain tracially AF \cstar-algebras due to Dadarlat. 
\end{abstract}

\maketitle


\section*{Introduction}

A separable \cstar-algebra $A$ is said to satisfy Rosenberg--Schochet's universal coefficient theorem (UCT) if for every separable C*-algebra $A'$, the following sequence is exact
\[
\xymatrix{
\Ext(K_*(A),K_{*-1}(A')) \ar@{^(->}[r] & KK_*(A,A') \ar@{->>}[r] & \Hom(K_*(A),K_*(A')),
}
\]
where the right hand map is the natural one and the left hand map is supposed to be the inverse of a map that is always defined; see \cite{RosenbergSchochet87}. 
A separable \cstar-algebra satisfies the UCT if and only if it is $KK$-equivalent to a commutative \cstar-algebra. 
Furthermore, the class of separable, nuclear \cstar-algebras satisfying the UCT can be characterized as the smallest class of separable, nuclear \cstar-algebras that contains $\IC$ and is closed under countable inductive limits, the two out of three property for extensions, and $KK$-equivalences; see \cite{BlaKK}. 
Whereas Skandalis \cite{Skandalis88} has shown that there exist non-examples within the class of separable, exact, non-nuclear \cstar-algebras, it is still open whether all separable, nuclear \cstar-algebras satisfy the UCT. 
This arguably constitutes the most important open question about separable, nuclear \cstar-algebras and is commonly referred to as the UCT problem.
Due to the recent dramatic progress in the structure and classification theory of simple, nuclear \cstar-algebras achieved by many hands --- see among others \cite{MatuiSato12, MatuiSato14UHF, GongLinNiu15, ElliottGongLinNiu15, TikuisisWhiteWinter17} --- the UCT problem is receiving an increasing amount of attention.

Over the years, different characterizations and reductions of the UCT problem have emerged. 
Kirchberg proved in \cite{KirchbergC} (and later \cite{Kirchberg04} in published form) that every separable, nuclear \cstar-algebra is $KK$-equivalent to some unital Kirchberg algebra, thus localizing the problem at the class of unital Kirchberg algebras. 
In fact, he showed that it is sufficient to only consider those Kirchberg algebras with vanishing $K$-theory, which by the Kirchberg--Phillips classification theory \cite{KirchbergC, Phillips00} then asks whether such \cstar-algebras must be isomorphic to the Cuntz algebra $\CO_2$; see \cite{Cuntz77,Cuntz81}.
Furthermore, by using either the work of Spielberg~\cite{Spielberg07_2} or combining results of Katsura \cite{Katsura08_2} and Yeend \cite{Yeend06,Yeend07}, one sees that UCT Kirchberg algebras have specific groupoid models.
These in turn give rise to Cartan subalgebras in the sense of Renault \cite{Renault08}. 
Conversely, it was recently recently shown in \cite{BarlakLi17} that any separable, nuclear \cstar-algebra with a Cartan subalgebra satisfies the UCT. 
As a consequence, the UCT problem turns out to be equivalent to the question whether every unital Kirchberg algebra has a Cartan subalgebra.

Using Kirchberg's aforementioned insight, the authors reduced the UCT problem in \cite{BarlakSzabo17} to crossed products of $\CO_2$ by certain actions of finite cyclic groups.
In addition to Kirchberg's previously known reduction theorem, the other two crucial ingredients of the proof were the Kirchberg--Phillips $\CO_2$-absor\-ption theorem \cite{KirchbergPhillips00} and the existence of certain model actions of $\IZ_p$ on $\CO_2$ for any prime number $p$.
The feature of any of these models is that the associated crossed product has, in a sense, the smallest possible non-trivial $K$-theory, namely that given by the $(p-1)$-fold direct sum of $M_{p^\infty}$; see also \cite{Izumi04, Izumi04II}\footnote{In these references, one sees that the possible $K$-groups for crossed products of $\CO_2$ by approximately representable $\IZ_p$-actions are precisely the modules over the ring $\IZ[\nicefrac 1p,e^{\nicefrac{2\pi i}{p}}]$. The additive group of this ring is in turn isomorphic to $\IZ[\nicefrac 1p]^{\oplus p-1}\cong K_0\big( M_{p^\infty}^{\oplus p-1} \big)$.}. 
The crossed product viewpoint has recently been taken up in \cite{BarlakLi17, BarlakLi17_2} to characterize the UCT problem in terms of the existence of invariant Cartan subalgebras for certain finite order automorphisms of $\CO_2$. 
We would also like to remark that in a similar fashion as for finite cyclic groups, a variant of the proof of \cite[23.15.12]{BlaKK} shows that the UCT problem can be reduced to crossed products of $\CO_2$ by certain actions of the circle group $\IT$ as well.

Building on Lin's classification of tracially AF \cstar-algebras \cite{Lin01, Lin04}, Dadarlat achieved in \cite{Dadarlat03} a reduction of the UCT problem to stably finite, nuclear \cstar-algebras, which is similar in spirit to Kirchberg's reduction theorem. 
He showed that the UCT problem has an affirmative answer if and only if the universal UHF algebra is the only separable, unital, simple, nuclear, tracially AF \cstar-algebra whose ordered $K$-theory is isomorphic to $(\IQ,0)$. 
As such \cstar-algebras are automatically monotracial, this in particular reduces the UCT problem to the class of separable, unital, simple, nuclear \cstar-algebras with a unique tracial state. 
Due to work of Sato--White--Winter \cite{SatoWhiteWinter15} and the recent classification of $KK$-contractible \cstar-algebras with finite nuclear dimension \cite{ElliottGongLinNiu17}, one may obtain a monotracial\footnote{Although we will not need it, we point out that using more advanced results around the Toms--Winter conjecture like \cite{BBSTWW} yields more general tracial versions in the same fashion.} analog of the Kirchberg--Phillips absorption theorem with the Razak--Jacelon algebra $\CW$ in place of $\CO_2$; see \cite{Razak02, Jacelon13} and Theorem \ref{thm:Robert}.

In this short note, we reduce the UCT problem to crossed products of $\CW$ by certain actions of finite cyclic groups. 
Similarly as in \cite{BarlakSzabo17}, we construct a model action of $\IZ_p$ on $\CW$ for each prime number $p$ such that the crossed product is in Robert's classifiable class \cite{Robert12}, and which is moreover monotracial and $KK$-equivalent to the $(p-1)$-fold direct sum of the UHF algebra $M_{p^\infty}$.  
The action itself arises as the dual action of another action constructed similarly as in \cite[Example 4.12]{BarlakSzabo17}, which we will carry out in detail for the reader's convenience. 
Once the model action is taken care of, our main result (Theorem~\ref{thm:main-result}) is then proved using the aforementioned $\CW$-absorption theorem for monotracial \cstar-algebras and Dadarlat's reduction theorem from \cite{Dadarlat03}.
Near the end of the paper, we will argue why our previous characterization of the UCT problem in terms of crossed products of $\CO_2$ can be directly recovered from the main result of this note, thus showing that we achieve a reduction that is a priori stronger.
It remains open whether the UCT problem can be further characterized in terms of the existence of certain Cartan subalgebras in $\CW$, akin to a similar phenomenon studied in \cite{BarlakLi17, BarlakLi17_2} for $\CO_2$.
This direction of research may be the subject of subsequent work.\bigskip

\textbf{Acknowledgement.}
Substantial parts of this work were carried out during research visits of the first author to the University of Aberdeen in July 2017 and to the University of Copenhagen in November 2017, respectively. 
He is grateful to these institutions for their hospitality and support.


\section{The model actions}

Throughout the paper, we will assume familiarity with $K$-theory and Kasparov's bivariant $KK$-theory; see \cite{BlaKK} for an introduction.
For standard references treating the Rokhlin property for finite group actions, see \cite{Izumi04, Izumi04II, Nawata16, Santiago15, GardellaSantiago16}.

\begin{defi}
Let $G$ be a finite group and $A$ a separable \cstar-algebra.
An action $\alpha: G\curvearrowright A$ is said to have the Rokhlin property, if there exists an approximately central sequence $e_n\in A$ of positive contractions satisfying
\[
\|(e_n-e_n^2)a\| + \Big\|\Big(\eins-\sum_{g\in G} \alpha_g(e_n) \Big)a \Big\| \ \stackrel{n\to\infty}{\longrightarrow} \ 0
\]
for all $a\in A$.
\end{defi}

\begin{rem}
Let $A$ be a finite \cstar-algebra, which has at least one tracial state but no unbounded traces.
Recall that an automorphism $\alpha$ on $A$ is called strongly outer, if it is outer, and for every $\alpha$-invariant tracial state $\tau$ on $A$, the induced automorphism of $\alpha$ on the von Neumann algebra $\pi_\tau(A)''$ is outer.
We will call an action $\alpha: G\curvearrowright A$ of a discrete group strongly outer, if $\alpha_g$ is a strongly outer automorphism for all $g\neq 1_G$.
Actions with the Rokhlin property are particular examples of strongly outer actions.
\end{rem}

\begin{defi}
Let $G$ be a finite abelian group and $A$ a separable \cstar-algebra.
An action $\alpha: G\curvearrowright A$ is said to be approximately representable, if there exist sequences of contractions $x_{g,n}\in A^\alpha$ in the fixed point algebra, for $g\in G$, such that the following properties hold for all $g,h\in G$ and $a\in A$:
\begin{itemize}
\item $(x_{1_G,n})_n$ is an approximate unit in $A$;
\item $\|a(x_{g,n}x_{h,n}-x_{gh,n})\|+\|(x_{g,n}x_{h,n}-x_{gh,n})a\|  \stackrel{n\to\infty}{\longrightarrow}  0$;
\item $\|a(x_{g,n}^*-x_{g^{-1},n})\|+\|(x_{g,n}^*-x_{g^{-1},n})a\|\stackrel{n\to\infty}{\longrightarrow} 0$;
\item $\|\alpha_g(a)-x_{g,n}ax_{g,n}^*\|\stackrel{n\to\infty}{\longrightarrow} 0$.
\end{itemize}
\end{defi}

\begin{rem} \label{rem:robert-class}
For what follows, we call Robert's class $\FC_R$ the class of separable \cstar-algebras that are stably isomorphic to inductive limits of 1-NCCW complexes with trivial $K_1$-groups on the level of building blocks.
We will make use of Robert's paper \cite{Robert12} where a classification theory for \cstar-algebras in $\FC_R$ is developed in terms of the functor $\Cu^\sim$ defined there.
If one restrict to the simple \cstar-algebras $A$ in $\FC_R$, then it is shown in \cite[Section 6]{Robert12} that $\Cu^\sim(A)$ is naturally isomorphic to $K_0(A)\sqcup\operatorname{LAff}^\sim_+(T_0(A))$.
Let us denote by $\FC_R^0$ the simple, stably projectionless \cstar-algebras in Robert's class with trivial pairing between the $K_0$-group and the traces.

In more recent work \cite{GongLin17} by Gong--Lin, specific models for simple, stably projectionless \cstar-algebras with trivial pairing maps are constructed.
In particular, one sees that for every pair $(G_0, T)$ of a torsion-free abelian group $G_0$ and Choquet simplex $T$, there exists a \cstar-algebra $A\in\FC_R^0$ with continuous scale and $(K_0(A), T(A))\cong (G_0, T)$; see \cite[Section 6]{GongLin17}.\footnote{The continuous scale assumption can be avoided if one replaces the Choquet simplex by a topological cone with a suitable scale.
However we do not need this level of generality for our applications.}

Moreover, it follows from Robert's aforementioned results that for every $A\in\FC_R^0$, and every group endomorphism $\kappa$ on $K_0(A)$, there exists a trace-preserving endomorphism $\beta$ on $A$ with $K_0(\beta)=\kappa$, and $\beta$ is unique up to approximate unitary equivalence.
By the usual Elliott intertwining method \cite[Corollary 2.3.4]{Rordam}, $\beta$ may be chosen to be an automorphism if $\kappa$ is an automorphism.
We will use this fact in the proof of the proposition below:
\end{rem}

The reader should now recall the Razak--Jacelon algebra $\CW$ from \cite{Razak02, Jacelon13}.

\begin{prop}[cf.\ {\cite[Example 4.12]{BarlakSzabo17}}] \label{prop:prime-example}
Let $p\geq 2$ be a prime number.
Then there exists a strongly outer, approximately representable action $\gamma: \IZ_p\curvearrowright\CW$ such that $\CW\rtimes_\gamma\IZ_p \sim_{KK} M_{p^\infty}^{\oplus p-1}$.
\end{prop}
\begin{proof}
Using the above Remark \ref{rem:robert-class}, we choose a \cstar-algebra $B\in\FC_R^0$ with a unique tracial state, no unbounded traces, and
\[
K_0(B) \cong \IZ[\nicefrac1p]^{\oplus p-1}.
\]
Since $B$ satisfies the UCT, it follows that $B\sim_{KK} M_{p^\infty}^{\oplus p-1}$, as the $K$-groups of these \cstar-algebras are isomorphic; see \cite{RosenbergSchochet87}.

As $p$ is prime, we note that the $K_0$-group is in fact isomorphic to the additive group of the ring $\IZ[\nicefrac1p, \xi_p]$, where $\xi_p = e^{\nicefrac{2\pi i}{p}}$.
Then there exists an automorphism $\beta$ on $B$ such that, under this identification, one has $K_0(\beta)(x) = \xi_p\cdot x$.

Similarly as in \cite{BarlakSzabo17}\footnote{The reference \cite[Section 5]{Nawata16} may be more appropriate for $p=2$.}, we will now reconstruct $B$ as a new inductive limit that will allow us to replace $\beta$ by a $\IZ_p$-action having the same invariant. 
Consider
\[
B_n = M_{p^{n-1}}\otimes B \quad\text{and}\quad \beta_n=\id_{M_{p^{n-1}}}\otimes\beta \quad\text{for } n\geq 1.
\]
From now on, we will use the index set $\IZ_p$ for the canonical matrix units in $M_p$, and identify $X^{\oplus p} \cong X^{\oplus \IZ_p}$ for any $X$ being either a \cstar-algebra or a group.
We will also denote $\beta_n^{\frak j} = \beta_n^j$ for all $n$, all $\frak j\in\IZ_p$ and $j\in\set{0,\dots,p-1}$ with $\frak j = j+p\IZ$.

We then define
\[
\Phi_n: B_n^{\oplus p} \to B_{(n+1)}^{\oplus p}
\]
via 
\[
\big[ \Phi_n(x_{\frak i})_{\frak i \in \IZ_p} \big]_{\frak j} = \diag\big( \beta_n^{\frak i}(x_{\frak i + \frak j}) \big)_{\frak i \in \IZ_p} 
\]
for $\frak j\in\IZ_p$.
Since this is just a combination of diagonal embeddings using compositions of $\beta_n$, we may describe this inductive limit on the level of $K$-theory by the following commutative diagram
\[
\xymatrix{
K_0( B_n^{\oplus p} ) \ar[d]_{\cong} \ar[rr]^{K_0(\Phi_n)} && K_0( B_{(n+1)}^{\oplus p} ) \ar[d]^{\cong} \\
\IZ[\nicefrac1p, \xi_p]^{\oplus p} \ar[rr]^{\phi} && \IZ[\nicefrac1p, \xi_p]^{\oplus p}
}
\]
where $\phi$ is given by multiplication with the $p \times p$-matrix
\[
\Xi_p = (\xi_p^{k-l})_{l,k=0}^{p-1} = \matrix{ 1 & \xi_p & \xi_p^2 & \cdots & \xi_p^{p-2} & \xi_p^{p-1} \\
\xi_p^{p-1} & 1 & \xi_p & \cdots & \xi_p^{p-3} & \xi_p^{p-2} \\
\vdots && \ddots && \vdots & \vdots \\
\vdots &&& \ddots & \vdots & \vdots \\
\xi_p^2 & \cdots & \cdots & \cdots & 1 & \xi_p \\
\xi_p & \cdots & \cdots & \cdots & \xi_p^{p-1} & 1
}.
\]
As the $K$-groups are uniquely $p$-divisible, using the simple relation
\[
\Xi_p^2 = p\cdot \Xi_p
\]
for this matrix yields that the inductive limit of the $K$-groups is isomorphic to the cokernel of $\phi$, which is easily seen to be isomorphic to $\IZ[\nicefrac1p, \xi_p] \cong \IZ[\nicefrac1p]^{\oplus p-1}$.

Now the limit is clearly separable, stably projectionless and is in Robert's class $\FC_R$.
As the connecting maps $\Phi_n$ are clearly non-degenerate, full, and trace-collapsing, it follows that the inductive limit
\[
\CB := \lim_{\longrightarrow} \{ B_n^{\oplus p}, \Phi_n \}
\]
is simple and has a unique tracial state, no unbounded traces and (thus automatically) a trivial pairing map.
As its $K$-theory is isomorphic to that of $B$, there exists an isomorphism $B\cong \CB$.
We will henceforth use this identification without much further mention.

This allows us to construct an inductive limit action $\alpha: \IZ_p\curvearrowright B \cong \CB$  arising on each building block from the action $\alpha^{(n)}: \IZ_p\curvearrowright B_n^{\oplus p}$ given by
\[
\alpha^{(n)}_{\frak j}(x_{\frak i})_{\frak i\in\IZ_p} = (x_{\frak i - \frak j})_{\frak i\in\IZ_p}.
\]
From the definition of the connecting maps, the formula
\[
\Phi_n\circ\alpha^{(n)}_{\frak j} = \alpha^{(n+1)}_{\frak j}\circ\Phi_n
\]
is evident for all $n$ and $\frak j\in\IZ_p$. In particular, $\alpha$ is well-defined.

Moreover, the standard identification $K_0( B_n^{\oplus p} ) \cong \IZ[\nicefrac1p, \xi_p]^{\oplus p}$ turns the $\IZ_p$-action $K_0(\alpha)$ into the canonical shift $\sigma: \IZ_p\curvearrowright\IZ[\nicefrac1p, \xi_p]^{\oplus p}$ on the level of abelian groups.
Using the above diagrams, this allows us to see that
\[
K_0(\Phi_n)\circ K_0(\alpha^{(n)}_1) = \phi\circ\sigma_1 = \Xi_p\cdot (\delta_{\frak i, \frak j+1})_{\frak i, \frak j \in\IZ_p} = \xi_p\cdot \Xi_p = (\xi_p\cdot\_\!\_\!\_)\circ K_0(\Phi_n).
\]
In other words, $K_0(\alpha_1)(x) = \xi_p\cdot x$ for all $x\in K_0(B)\cong \IZ[\nicefrac1p, \xi_p]$.

From the construction of $\alpha$, it is clear that it has the Rokhlin property.
In fact on every building block, one has a canonical equivariant embedding from $\CC(\IZ_p)$ equipped with the shift into the center of $\mathcal M( B_n^{\oplus p} )$, and so $\alpha$ has the Rokhlin property arising as an equivariant inductive limit of Rokhlin actions.
Applying \cite[Theorem 3]{Santiago15}, we deduce that $B\rtimes_\alpha\IZ_p$ is a \cstar-algebra in the class $\FC_R^0$.
It clearly has a unique and bounded trace, and is Morita equivalent to the fixed point algebra $B^\alpha$, whose $K$-theory is naturally isomorphic to\footnote{See \cite[Theorem 4.9]{BarlakSzabo16ss}. In the simple unital case, this observation is \cite[Theorem 3.13]{Izumi04}.}
\[
\{ x\in K_*(B) \mid K_*(\alpha_1)(x) = \xi_p\cdot x = x \} = 0.
\]
In summary, we deduce that $B\rtimes_\alpha\IZ_p\cong\CW$.
We obtain the action $\gamma: \IZ_p\curvearrowright\CW$ as the dual $\gamma=\hat{\alpha}$ under this identification, which has the desired property by Takai duality \cite{Takai75}.
As $\alpha$ has the Rokhlin property, $\gamma$ will be approximately representable by \cite[Proposition 4.4]{Nawata16}; in fact locally $B$-representable by \cite[Theorem 3.4]{BarlakSzabo17}.
It is also strongly outer as the crossed product, which is isomorphic to $B$, has a unique trace. This finishes the proof.
\end{proof}

\begin{rem} \label{rem:shorter-proof}
We note that most of the above proof of Proposition \ref{prop:prime-example} is given in detail only in order for this note to be more self-contained.
In fact, right after choosing $\beta\in\Aut(B)$ near the beginning of the proof, one can deduce from Remark \ref{rem:robert-class} that $\beta$ defines a $\IZ_p$-action up to approximate unitary equivalence.
Due to classification, one has $B\cong B\otimes M_{p^\infty}$ as the $K$-groups of $B$ are uniquely $p$-divisible.
Applying the much more general existence result \cite[Theorem 2.3]{BarlakSzabo17} would then immediately give the Rokhlin action $\alpha:\IZ_p\curvearrowright B$ with $\alpha_1\ue\beta$, which implies $K_0(\alpha_1)(x)=\xi_p\cdot x$ and allows one to finish the proof in the same manner as above.
\end{rem}


\section{Main theorem}

The following insight due to Dadarlat forms the basis of our main result.
See \cite{Lin01, Lin04} for details surrounding tracially AF \cstar-algebras.

We note that the equivalence \ref{thm:dadarlat:1}$\Leftrightarrow$\ref{thm:dadarlat:2} is actually proved in Dadarlat's work, whereas the equivalence \ref{thm:dadarlat:2}$\Leftrightarrow$\ref{thm:dadarlat:3} follows from the simple fact that any \cstar-algebra as in \ref{thm:dadarlat:2} is automatically monotracial.
We will only use the equivalence \ref{thm:dadarlat:1}$\Leftrightarrow$\ref{thm:dadarlat:3} in the sequel.

\begin{theorem}[see {\cite[Theorem 1.2]{Dadarlat03}}] \label{thm:dadarlat}
The following are equivalent:
\begin{enumerate}[label=\textup{(\roman*)}, leftmargin=*]
\item The UCT holds for every separable, nuclear \cstar-algebra; \label{thm:dadarlat:1}
\item if $A$ is any separable, unital, simple, nuclear, tracially AF \cstar-algebra such that one has an order-isomorphism $(K_0(A),K_1(A))\cong (\IQ,0)$, then $A$ is isomorphic to the universal UHF algebra;\label{thm:dadarlat:2}
\item the UCT holds for every separable, unital, simple, nuclear \cstar-algebra with a unique tracial state. \label{thm:dadarlat:3}
\end{enumerate}
\end{theorem}

The following arises as an application of the recent advances in the structure theory of simple \cstar-algebras.
This should be understood as a monotracial variant of the Kirchberg--Phillips $\CO_2$-absorption theorem \cite{KirchbergPhillips00}.

\begin{theorem} \label{thm:Robert}
Let $A$ be a separable, unital, simple, and nuclear \cstar-algebra. Then $A$ has a unique tracial state if and only if $A\otimes\CW\cong\CW$.
\end{theorem}
\begin{proof}
As $\CW$ is $\CZ$-stable, one has $A\otimes\CW \cong (A\otimes\CZ)\otimes\CW$, and hence we may assume that $A$ is $\CZ$-stable.

Let us show the ``if'' part.
We may assume that $A$ has tracial states. For if $A$ is traceless, then it is a Kirchberg algebra --- see \cite{Rordam04} --- and thus $A\otimes\CW\cong\CO_2\otimes\CK$ by Kirchberg--Phillips classification \cite{Phillips00} as this tensor product is a $KK$-contractible, non-unital Kirchberg algebra.
If $A$ has more than one trace, then this immediately gives rise to distinct tracial states on $A\otimes\CW$, so it cannot be isomorphic to $\CW$, which has a unique trace. Thus $A$ is indeed monotracial.

For the ``only if'' part, we assume that $A$ is monotracial.
As it is also $\CZ$-stable by assumption, it follows from the main result of \cite{SatoWhiteWinter15} that $A$ has finite nuclear dimension.
In particular, the tensor product $A\otimes\CW$ also has finite nuclear dimension.
But then it follows from \cite[Theorem 7.5]{ElliottGongLinNiu17} that $A\otimes\CW$ is classifiable.
As it is $KK$-contractible, has a unique tracial state and no unbounded traces, it must be isomorphic to $\CW$.
\end{proof}

We now come to the main result of this note, which is a finite analog of \cite[Theorem 4.17]{BarlakSzabo17}.

\begin{theorem} \label{thm:main-result}
Let $q_1, q_2\geq 2$ be two distinct prime numbers.
Then the following are equivalent:
\begin{enumerate}[label=\textup{(\roman*)}, leftmargin=*, resume]
\item The UCT holds for all separable, nuclear \cstar-algebras; \label{main-result:1}
\item for $p\in\set{q_1, q_2}$ and for all strongly outer, approximately representable actions $\alpha: \IZ_p\curvearrowright\CW$, the crossed product $\CW\rtimes_\alpha\IZ_p$ satisfies the UCT. \label{main-result:2}
\end{enumerate}
\end{theorem}
\begin{proof}
The implication \ref{main-result:1}$\Rightarrow$\ref{main-result:2} is obviously trivial.

\ref{main-result:2}$\Rightarrow$\ref{main-result:1}:
Suppose that the UCT fails for some separable, nuclear \cstar-algebra.
Then by Theorem \ref{thm:dadarlat}, it fails for a separable, unital, simple, nuclear \cstar-algebra $A$ with a unique tracial state.
There exists a natural short exact sequence
\[
\xymatrix{
\CC_0\big( \IR,A\otimes M_{(q_1q_2)^\infty} \big) \ar@{^(->}[r] & A\otimes Z_{q_1^\infty,q_2^\infty} \ar@{->>}[r] & A\otimes (M_{q_1^\infty}\oplus M_{q_2^\infty}).
}
\]
Since $Z_{q_1^\infty,q_2^\infty}\sim_{KK}\IC$, the UCT must fail for $A\otimes M_{q_1^\infty}$ or $A\otimes M_{q_2^\infty}$.
In particular, we may choose $p\in\set{q_1,q_2}$ and assume that $A\cong A\otimes M_{p^\infty}$.

Let $\gamma: \IZ_p\curvearrowright \CW$ be the action from Proposition \ref{prop:prime-example}.
Then the UCT fails for $A\otimes M_{p^\infty}^{\oplus p-1} \sim_{KK} A\otimes (\CW\rtimes_\gamma\IZ_p)$.
Theorem \ref{thm:Robert} implies $A\otimes\CW\cong\CW$.
So if use this identification and set $\alpha=\id_A\otimes\gamma: \IZ_p\curvearrowright A\otimes\CW\cong\CW$, then we get that
\[
\CW\rtimes_\alpha\IZ_p \cong A\otimes (\CW\rtimes_\gamma\IZ_p) \sim_{KK} A\otimes M_{p^\infty}^{\oplus p-1} \cong A^{\oplus p-1}
\]
does not satisfy the UCT.
Clearly $\alpha$ is strongly outer and approximately representable as $\gamma$ had these properties.
This finishes the proof.
\end{proof}

Let us now explain how Theorem \ref{thm:main-result} is a priori stronger than our previous characterization of the UCT problem \cite{BarlakSzabo17} in terms of crossed products on $\CO_2$.

\begin{rem} \label{rem:a-priori-stronger}
Let $p\geq 2$ be a prime number.
If all crossed products of the form $\CO_2\rtimes_\alpha\IZ_p$ satisfy the UCT, then so do all crossed products of the form $\CW\rtimes_{\alpha'}\IZ_p$.\footnote{Although we do not need it, the argument works in fact for arbitrary finite groups in place of $\IZ_p$.}
The same is true if we restrict the statement to actions $\alpha: \IZ_p\curvearrowright\CO_2$ and $\alpha':\IZ_p\curvearrowright\CW$ that are assumed to be outer and/or approximately representable.
In particular, Theorem \ref{thm:main-result} in combination with \cite[Theorem 3.4]{BarlakSzabo17} formally implies \cite[Theorem 4.17 (1)$\Leftrightarrow$(3)]{BarlakSzabo17}.
\end{rem}
\begin{proof}
Let $\alpha': \IZ_p\curvearrowright\CW$ be an action.
It is well-known that the unital inclusion $\IC\subset\CO_\infty$ is a $KK$-equivalence; see \cite{Cuntz81}.
Then $\alpha'\otimes\id_{\CO_\infty}: \IZ_p\curvearrowright\CW\otimes\CO_\infty$ is an action on $\CW\otimes\CO_\infty\cong\CO_2\otimes\CK$ such that its crossed product satisfies the UCT precisely when $\CW\rtimes_{\alpha'}\IZ_p$ satisfies the UCT.

Now let $\beta: \IZ_p\curvearrowright\CO_2\otimes\CK$ be some arbitrary action.
As all non-trivial projections in $\CO_2 \otimes \CK$ are equivalent, there exists a unitary $z\in\CU(\CM(\CO_2\otimes\CK))$ with $\ad(z)\circ\beta\circ(\eins_{\CO_2}\otimes\id_\CK)=\eins_{\CO_2}\otimes\id_\CK$.
We denote $z_{\mathfrak j}=z\beta(z)\cdots\beta^{j-1}(z)$ for $\mathfrak j\in\IZ_p$, where $j=0,\dots,p-1$ with $\mathfrak j=j+p\IZ$.
Note that this yields the formula $\ad(z_{\frak j})\circ\beta_{\frak j}\circ(\eins_{\CO_2}\otimes\id_\CK)=\eins_{\CO_2}\otimes\id_\CK$ for all $\frak j\in\IZ_p$.

Let us consider the family $\set{w_{\frak i, \frak j}}_{\frak i,\frak j\in\IZ_p}$ in $\CU(\CM(\CO_2\otimes\CK))$ given by $w_{\frak i,\frak j} = z_{\frak i}\beta_{\frak i}(z_{\frak j})z^*_{\frak i+\frak j}$. 
Then we obtain a cocycle action $(\beta',w): \IZ_p\curvearrowright\CO_2\otimes\CK$ via $\beta'_{\frak j}=\ad(z_{\frak j})\circ\beta_{\frak j}$.
By definition, $\beta$ is exterior equivalent to $(\beta',w)$.
It is also immediate from the construction that $\beta'$ fixes elements in $\eins_{\CO_2}\otimes\CK$ pointwise, and that the family $\set{w_{\frak i,\frak j}}$ commutes with $\eins_{\CO_2}\otimes\CK$.
Hence $(\beta',w)$ restricts to a cocycle action 
\[
(\gamma,v): \IZ_p\curvearrowright\CO_2=\CO_2\otimes\eins =\CM(\CO_2\otimes\CK)\cap (\eins_{\CO_2}\otimes\CK)'
\]
such that $(\beta',w)=(\gamma\otimes\id_\CK,v\otimes\eins)$; see \cite[Proposition 1.7]{BarlakSzabo16ss}.
Using the Packer--Raeburn 2-cocycle vanishing trick \cite[Theorem 3.4]{PackerRaeburn89}, there exists some genuine action $\delta: \IZ_p\curvearrowright \CO_2\otimes M_p\cong\CO_2$ that is cocycle conjugate to $(\gamma\otimes\id_{M_p},v\otimes\eins_{M_p})$.
To summarize, the action $\beta$ is equivalently Morita equivalent to $\delta$, and on the level of crossed products we obtain
\[
\begin{array}{ccl}
(\CO_2\otimes\CK)\rtimes_\beta\IZ_p &\cong& (\CO_2\otimes\CK)\rtimes_{\beta',w}\IZ_p \\
&\cong& (\CO_2\rtimes_{\gamma,v}\IZ_p)\otimes\CK \\
&\cong& (\CO_2\rtimes_{\gamma,v}\IZ_p)\otimes M_p\otimes\CK \\
&\cong& (\CO_2\rtimes_\delta\IZ_p)\otimes\CK.
\end{array}
\]
The claim now follows when we combine all steps so far.
The second part of the statement follows as the class of all outer and/or approximately representable actions is invariant under tensoring with $\id_{\CO_\infty}$ and under equivariant Morita equivalence; see \cite[Propositions 3.14 and 4.23]{BarlakSzabo16ss} for details.
\end{proof}

\begin{rem}
It is unclear whether there is any direct formal converse to Remark \ref{rem:a-priori-stronger}.
In any case, it is not hard to see that the assignment $\alpha'\mapsto\alpha'\otimes\id_{\CO_\infty}$ for outer actions $\alpha': \IZ_p\curvearrowright\CW$ does not remember all the dynamical information.
In fact, let us briefly sketch how to obtain two non-cocycle conjugate actions the images of which become cocycle conjugate under this assignment.
Consider $\gamma$ to be the action constructed in Proposition \ref{prop:prime-example}.

Similarly as in the proof of Proposition \ref{prop:prime-example}, we may consider $D$ to be the \cstar-algebra in $\FC_R^0$ with continuous scale, $K_0(D)\cong\IZ[\nicefrac 1p,\xi_p]$, and two extremal tracial states $\tau_0$ and $\tau_1$.
By Robert's classification, there exists $\beta\in\Aut(D)$ satisfying $K_0(\beta)=\xi_p\cdot\_\!\_\!\_$ and $\tau_i\circ\beta=\tau_{1-i}$. 
This allows one to run the rest of the argument in the completely analogous fashion (or to proceed as suggested in Remark \ref{rem:shorter-proof}), and obtain an approximately representable action $\delta: \IZ_p\curvearrowright\CW$ with $\CW\rtimes_\delta\IZ_p\cong D$.
Clearly $\gamma$ and $\delta$ are not cocycle conjugate as their crossed products are  non-isomorphic. 
In fact, they have $KK$-equivalent crossed products, but $\gamma$ is strongly outer while $\delta$ is not.

However, the actions $\gamma\otimes\id_{\CO_\infty}$ and $\delta\otimes\id_{\CO_\infty}$ can be identified with approximately representable actions on $\CW\otimes\CO_\infty\cong\CO_2\otimes\CK$ whose crossed products are isomorphic.
Since the dual actions act the same way on $K$-theory, namely via multiplication by $\xi_p$ on the $K_0$-group $\IZ[\nicefrac 1p, \xi_p]$, these actions turn out to be cocycle conjugate; see \cite{Izumi04, Izumi04II}.
\end{rem}

We end this note with two questions, which are motivated by analogous known problems about the Cuntz algebra $\CO_2$.

\begin{question}
Is every action of a finite group on the Razak--Jacelon algebra $\CW$ approximately representable?
What about actions of countable discrete groups?
\end{question}

\begin{question}[cf.\ {\cite{BarlakLi17, BarlakLi17_2}}]
Let $p\geq 2$ be a prime number and let $\alpha: \IZ_p\curvearrowright\CW$ be a strongly outer, approximately representable action.
Does $\CW\rtimes_\alpha\IZ_p$ satisfy the UCT precisely when $\alpha$ fixes a Cartan subalgebra of $\CW$ globally?
\end{question}

\bibliographystyle{gabor}
\bibliography{master2}
\end{document}